\documentclass[a4papersize]{article}

\usepackage{amsmath,amsthm, amsxtra,amssymb,latexsym, amscd}
\usepackage[mathscr]{eucal}
\newtheorem{theorem}{Theorem}[section]
\newtheorem{corollary}[theorem]{Corollary}
\newtheorem{lemma}[theorem]{Lemma}

\newtheorem{example}[theorem]{Example}

\newtheorem{remark}[theorem]{Remark}

\DeclareMathOperator{\depth}{depth}

\DeclareMathOperator{\Hom}{Hom}
\DeclareMathOperator{\Ann}{Ann}

\DeclareMathOperator{\Ass}{Ass}
\DeclareMathOperator{\Supp}{Supp}
\DeclareMathOperator{\Assh}{Assh}
\DeclareMathOperator{\Var}{Var}

\DeclareMathOperator{\Att}{Att}

\DeclareMathOperator{\Spec}{Spec}

\begin{document}
	\large
	\centerline{\Large {\bf ON THE TRANSFER OF ARTINIANESS AND  }}
	\smallskip
	\centerline{\Large {\bf COHEN-MACAULAYNESS  UNDER MODULE-FINITE  EXTENSIONS}}
	\smallskip

	\medskip
	\vskip 0.7cm
	
	\centerline {TRAN DO MINH CHAU}
	\centerline { Thai Nguyen University of Education, Thai Nguyen, Vietnam}
	\centerline {E-mail: chautdm@tnue.edu.vn}
	
	\vskip 0.6cm

	\centerline { Dedicated to Professor Nguyen Tu Cuong  on the occasion of his 75th birthday.}
	
	\vskip 1cm
	
\noindent{\bf Abstract} {\footnote{ {\it{Key words and phrases: }} Module-finite extension; local cohomology module; idealization; Cohen-Macaulay  module. \hfill\break{\it{2020 Subject  Classification: }} 13B02, 13D45, 13E05,  13E10.}
		
This paper deals with certain classes of modules under module-finite extensions.  Let  $\varphi: R\hookrightarrow S$  be a module-finite extension between commutative Noetherian local rings. We investigate  the transfer of Artinian module structures and attached primes between $R$ and $S$. We clarify the behavior of local cohomology modules as well as the structure of finitely generated $S$-modules under the restriction of scalars to $R$ via $\varphi$. We show that $R$ is a quotient of a Cohen-Macaulay local ring if and only if so is $S$. As an application, we characterize the structure of Nagata's idealization. Using Macaulayfication of  algebraic varieties and idealization, we give an example to illustrate the results.	 
			 		
\section{Introduction}
		
\ \ \ \  Throughout this paper, let $(R,\frak m )$ and $(S,\frak n )$ be  Noetherian local rings. Let  $\varphi: R\hookrightarrow S$  be a module-finite extension, i.e. $\varphi$ is a ring monomorphism and $S$ is  finitely generated as an $R$-module defined by $\varphi$. 
				
Module-finite extensions appear in various areas of commutative algebra and algebraic geometry with significant applications. The  Direct Summand Conjecture  formulated by M. Hochster \cite{Ho1}  asserts that regular local rings are direct summands of their module-finite extensions. This conjecture has been known to be equivalent to other important homological conjectures.  The Direct Summand Conjecture was proved by M. Hochster \cite{Ho1}  in equicharacteristic case and by Y. Andr\'{e} \cite{And} in mixed characteristic case based on Peter Scholze's theory of perfectoid spaces. The Noether Normalization Theorem \cite{EM} ensures that any finitely generated algebra over a field $K$ can be viewed as   a module-finite extension of a  polynomial ring over $K$.  Thus, Noether Normalization Theorem provides the structure of finitely generated algebra over a field, while the above Hochster's result provides the splitting property for that structure.		

An important problem concerning module-finite extension is whether  there exists a Cohen-Macaulayfication of a local ring as well as the Macaulayfication of an algebraic variety. There are some  criteria for the existence of a Cohen-Macaulayfication of a local ring $R$ (i.e. a Cohen-Macaulay intermediate  ring $S$ between $R$ and the total  quotient ring $Q(R)$ such that  $R\hookrightarrow S$ is a module-finite extension), see \cite{AG, Sch1}.

A significant module-finite extension is the canonical injection $R\hookrightarrow R\ltimes M$, where  $R\ltimes M$ is the  idealization of an $R$-module $M$ over $R$ introduced by M. Nagata \cite{Na}. The notion of idealization  is  a powerful technique that turns a module into an ideal  within a larger ring, which helps to construct examples of  rings with  specific properties. Idealization is  a crucial tool in deformation theory, homological algebra and local algebra. For example, idealization is used to give a criterion for a local ring being a quotient of a Gorenstein local ring. Concretely, I. Reiten \cite{Re} proved that if $R$  is Cohen-Macaulay and $M\neq 0$ is a finitely generated $R$-module, then $R\ltimes M$  is Gorenstein if and only if $M$ is a Gorenstein module of rank $1.$ In this case, $R$ is a quotient of the Gorenstein local ring $R\ltimes M.$ 
	
In this paper, we study Artinian module structures, local cohomology modules and certain  classes of finitely generated modules  under module-finite extension $\varphi:R\to S.$   We investigate  the transfer of Artinian module structures and attached primes between $R$ and $S$.  We clarify the behavior of local cohomology modules as well as certain structures of finitely generated $S$-modules under the restriction of scalars to $R$ via $\varphi$. We show that $R$ is a quotient of a Cohen-Macaulay local ring if and only if so is $S$. As an application, we characterize the structure of the idealizations. 
	 
 The following theorem, which is the first main result of this paper, gives results on the Artinian module structures as well as attached primes under the effect of $\varphi$. If $A$ is an Artinian $R$-module,  set $\dim_R(A)=\dim(R/\Ann_RA).$    For a finitely generated $R$-module $M$, set $$\Assh_R(M)=\{\frak p\in\Ass_R(M)\mid \dim(R/\frak p)=\dim_R(M)\}.$$

\begin{theorem}\label{T:1a} The following statements are true. 
\begin{itemize}

\item [{\rm (a)}]  Let	 $B$  be  an $S$-module. Then $B$ is Artinian as an $S$-module if and only if $B$ is   Artinian as an $R$-module defined by $\varphi$. In this case we have $\dim_{S}(B)=\dim_R(B)$ and	$$\Att_S(B)\subseteq\bigcup_{\frak p\in\Att_R(B)}\Assh_S(S/\frak pS).$$	
	
\item [{\rm (b)}] If $A$ is an Artinian $R$-module then $A\otimes_RS$ is an Artinian $S$-module. In this case we have $\dim_S(A\otimes_RS)=\dim_R(A)$ and  					
$$\Att_S(A\otimes_RS)\subseteq\bigcup_{\frak p\in\Att_R(A)}\Assh_S(S/\frak pS).$$ 			
\item [{\rm (c) }]  $\Att_S(B)=\bigcup_{\frak p\in\Att_R(B)}\Assh_S(S/\frak pS)$ for every Artinian $S$-module $B$ if and only if the induced map $\varphi^*:\Spec(S)\to\Spec(R)$ sending $\frak P$ to $\frak P\cap R$ is a bijection.				
\end{itemize}

\end{theorem}
From Theorem \ref{T:1a}(b), it is natural to ask if  $A$ is an Artinian $R$-module provided $A\otimes_RS$ is  Artinian. Lemma \ref{L:5}(b) together with Hochster's Direct Summand Conjecture shows that if either $R$ is a regular ring or $S$ is an idealization of a finitely generated module over $R$, then the answer to this question is  affirmative.

 The notion of Cohen-Macaulay module is a central object of research in commutative algebra. Some important extensions of the class of    Cohen-Macaulay modules are the classes of   generalized Cohen-Macaulay modules,  $1$-Cohen-Macaulay modules,  sequentially  Cohen-Macaulay modules, sequentially  generalized Cohen-Macaulay modules and sequentially  $1$-Cohen-Macaulay modules, cf. \cite{CSSS, CN, CST, LLN, LN, NDC, Sch, St}.   
 
 The following theorem, which is the second main result of this paper, describes  the behavior of local cohomology modules as well as characterizes the above classes of modules under module-finite extensions. Moreover,  the property of being a quotient of a Cohen-Macaulay local ring can  transfer from $S$ to $R$ as well as from $R$ to $S.$  
 
\begin{theorem}\label{T:2a} Let $N$ be a finitely generated $S$-module. The following statements hold true. 
\begin{itemize}			
\item [{\rm (a)}] For each integer $i$, there is an isomorphism $H^i_{\frak n}(N)\cong  H^i_{\frak m}(N)$ of Artinian $R$-modules. Moreover, if $H^i_{\frak n}(N)\neq 0$ then  
$\dim_S(H^i_{\frak n}(N))=\dim_R(H^i_{\frak m}(N)).$
\item [{\rm (b)}] $N$ is Cohen-Macaulay (resp. generalized Cohen-Macaulay, $1$-Cohen-Macaulay,  sequentially  Cohen-Macaulay, sequentially  generalized Cohen-Macaulay, sequentially  $1$-Cohen-Macaulay)  as an $S$-module if and only if so is $N$ as an $R$-module.
\item [{\rm (c)}]  $R$ is a quotient of a Cohen-Macaulay local ring if and only if so is  $S$.
\end{itemize} 		
\end{theorem}

As an application of the main results, we  characterize the Cohen-Macaulayness, generalized Cohen-Macaulayness, $1$-Cohen-Macaulayness, sequentially Cohen-Macaulayness,  sequentially generalized  Cohen-Macaulayness, sequentially $1$-Cohen-Macaulayness for idealizations (Corollaries \ref{C:1}, \ref{C:2}). Using Macaulayfication of  algebraic varieties and idealization, we give an example to illustrate the results (Example \ref{E:2}).
 
 In the next section, we give some preliminaries that will be used in the sequel. The proofs of the main results, Theorems \ref{T:1a}, \ref{T:2a}, are presented in the last section.   

\section{ Preliminaries}	
		
\ \ \ \  Firstly, we recall Melkersson's criterion \cite[Theorem 1.3]{Me} for Artinian modules. Let $A$ be an $R$-module and $I$ an ideal of $R$. We say that $A$ is {\it $I$-torsion} if $A=\bigcup_{n\in\Bbb N}(0:_AI^n).$
\begin{lemma}\label{L:1}
	Let $A$ be an $R$-module. Then $A$ is an Artinian $R$-module if and only if there exists an ideal $I$ of $R$ such that $A$ is $I$-torsion and $(0:_AI)$ is Artinian.   
\end{lemma}

I. G. Macdonald \cite{Mac} introduced a theory of secondary representation for modules.  Let $A$ be an $R$-module and $\mathfrak p\in\Spec(R)$. We say that $A$ is {\it $\mathfrak p$-secondary} if $A\neq 0$ and the multiplication by $r$ on $A$ is  surjective for all $r\in R\setminus \mathfrak p$ and  it is nilpotent for all $r\in\mathfrak p$. We say that $A$ is {\it representable} if either $A=0$ or $A$ has a minimal secondary representation $A=A_1+\ldots+A_n$, where each $A_i$ is a non-redundant $\frak p_i$-secondary component and $\frak p_i\ne\frak p_j$ for all $i\ne j.$ The set $\{\frak p_1,\ldots,\frak p_n\}$ is independent on the choice of the minimal secondary representation of $A$. It is called the set of {\it attached primes} of $A$ and denoted by $\Att_R(A).$ 
	
We recall some properties of attached primes, see \cite{BS, Mac, NDC}. For each ideal $I$ of $R$, denote by $\Var(I)$ the set of all prime ideals of $R$ containing $I.$ 		  
			 
\begin{lemma}\label{L:2}  Let $A$ be an Artinian $R$-module,  $B$ a representable $S$-module and $M$ a finitely generated $R$-module.  Then 
\begin{itemize}
\item [{\rm (a)}] $A$ is representable and $\min\Att_R(A)=\min\Var (\Ann_RA).$ In particular, $$\dim_R(A)=\max\{\dim(R/\frak p)\mid\frak p\in \Att_R(A)\}.$$
				
\item [{\rm (b)}] $A\neq 0$ if and only if $\Att_R(A)\neq \emptyset.$ If $A\neq 0$ then $\ell_R(A)<\infty$ if and only if $\dim_R(A)=0$,  if and only if $\Att_R(A)=\{\frak m\}.$ 
			
\item [{\rm (c)}]  $B$ is representable as an $R$-module defined by $\varphi$ and $$\Att_R(B)=\{\frak P\cap R\mid \frak P\in\Att_S(B)\}.$$
	 
\item [{\rm (d)}] $H^i_{\frak m}(M)$ is an Artinian $R$-module for every integer $i$.  If $\frak p\in\Ass_R(M)$ and $\dim (R/\frak p)=i$ then $\frak p\in\Att_R(H^i_{\frak m}(M)).$ If $d=\dim_R(M)$ then
$$\Att_R(H^d_{\mathfrak m}(M))=\{\mathfrak p \in \Ass_R(M) \mid \dim (R/\mathfrak p)=d\}.$$
	 	
\item [{\rm (e)}] Suppose that $R$ is a quotient of a Cohen-Macaulay local ring and $\frak p\in\Att_R(H^i_{\frak m}(M))$. Then $\dim(R/\frak p)\leq i$. 
\end{itemize}
\end{lemma}
			
We describe the non-vanishing and the attached primes of a tensor product of $S$ and an Artinian $R$-module. 
			
\begin{lemma}\label{L:2a} Let $A$ be an Artinian $R$-module. Then 
 $A\otimes_RS$ is an Artinian $R$-module and $$\Att_R(A\otimes_RS)=\Att_R(A).$$ In particular, $A\neq 0$ if and only if $A\otimes_RS\neq 0$.
\end{lemma}

\begin{proof} Since $S$ is finitely generated as an $R$-module,  $A\otimes_RS$ is an Artinian $R$-module. As $\varphi$ is injective, we have $\Supp_R(S)=\Spec(R).$ Therefore, we get  by \cite[Proposition 5.2]{MSch} that 
$$\Att_R(A\otimes_RS)=\Att_R(A)\cap \Supp_R(S)=\Att_R(A).$$
The remaining statement follows   by Lemma \ref{L:2}(b). 
\end{proof}			
			 
The following example shows that monomorphisms between Artinian modules are not preserved under  module-finite extensions in general.  Moreover, we cannot embed an Artinian $R$-module $A$ into $A\otimes_RS$.  
 
\begin{example} \label{E:1} {\rm Let $S=K[[t]]$ be the formal power series ring of one variable over a field $K$ and $R=K[[t^2, t^3]]$. Let $\mathfrak{m}=(t^2, t^3)$, the unique maximal ideal of $R$. Let   $A=R/\mathfrak{m}^2$, $A'=\mathfrak{m}/\mathfrak{m}^2$, $E= E_R(R/\frak m)$ the injective hull of the residue field of $R$. 
Then $R\hookrightarrow S$ is a module-finite extension, $A, A', E$ are Artinian $R$-modules, the induced homomorphism $f^*: A'\otimes_R S\to A\otimes_R S$   is not injective, and the homomorphism $E\to E\otimes_RS$ sending $a$ to $a\otimes 1$ is not  injective}.
\end{example}

\begin{proof} Note that $A'$ is a vector space over  $R/\frak m$ with a basis $\{t^2, t^3\}$ and  $S/\mathfrak{m}S$ is a vector space over  $R/\frak m$ with a basis $\{1, t\}$. Since $\mathfrak{m}A'=0,$ we have $$A' \otimes_RS \cong A' \otimes_{R/\mathfrak{m}} (S/\mathfrak{m}S).$$ Hence   $A' \otimes_R S$ is a vector space over  $R/\frak m$ of dimension  $4$ with a basis $\{t^2 \otimes 1, t^2 \otimes t, t^3 \otimes 1, t^3 \otimes t\}$. As $t^3 \otimes t$ is an element in  this  basis, $t^3 \otimes t \neq 0$ in  $A' \otimes_R S$. Note that    
$$A \otimes_R S \cong S/\mathfrak{m}^2 S\cong K[[t]]/(t^4).$$ So, $t^4(A \otimes_R S)=0$. Since  $f^*(t^3 \otimes t) = t^3 \otimes t = t^3(1\otimes t)=t^4(1\otimes 1),$ we have $f^*(t^3 \otimes t)=0$ in $A \otimes_R S.$ Therefore, $f^*$ is not injective.

Now we prove that the homomorphism $E \to E \otimes_R S$ is not an injection of $R$-modules. Suppose on the contrary that this homomorphism  is injective. Note that  $\Hom_R(E, E)\cong  R$ and $$\text{Hom}_R(E \otimes_R S, E)\cong \text{Hom}_R(S, \Hom_R(E, E))\cong \Hom_R(S, R).$$ So, applying Matlis duality to $E \to E \otimes_R S$   provides a surjection $\Hom_R(S, R) \to R$ of $R$-modules sending $g$ to $g(1).$  Since $1\in R$, there exists   $g \in \Hom_R(S, R)$ such that $g(1)=1.$  Note that $t\in S$ and $t^2, t^3\in R$, so we have
$$t^2 g(t) = g(t^3) = t^3 g(1) = t^3.$$ As $R$ is a domain, we have $g(t) = t\notin R$, this gives a contradiction. 
\end{proof}
 		
We recall some basic properties of module-finite extensions. For an ideal $I$ of $R$ and an ideal  $\frak J$ of $S$, the extension ideal $\varphi(I)S$ of $I$ is denoted by $IS,$ the contraction ideal $\varphi^{-1}(\frak J)$ of $\frak J$ is denoted by $\frak J\cap R$.  

\begin{remark}\label{R:3} {\rm Since $\varphi: R\hookrightarrow S$ is a module-finite extension,  $S$ is finitely generated as an $R$-algebra and $S$ is integral over $R$. So, we can find  the following facts in   \cite[Chapter 5]{AM}.}
\begin{itemize}
\item [{\rm (a)}]  {\rm The induced map $\varphi^*:  \Spec(S)\to \Spec(R)$ sending $\frak P$ to $\frak P\cap R$ is a surjection. In particular, $\frak n\cap R=\frak m$ and $\frak mS\subseteq\frak n$}.
\item [{\rm (b)}] {\rm $\varphi$ satisfies the going-up property.} 
\item [{\rm (c)}] {\rm Let $\frak P, \frak Q\in\Spec (S)$. If $\frak P\subseteq \frak Q$ and $\frak P\cap R=\frak Q\cap R$, then $\frak P=\frak Q.$}
\end{itemize}
\end{remark}

Using Remark \ref{R:3}, we get the following properties of module-finite extensions. These properties should also be well-known somewhere. For convenience, we provide a proof here. 
	
\begin{lemma}\label{L:3a} The following statements are true.
\begin{itemize}
\item [{\rm (a)}] $\dim (R)=\dim (S).$
\item [{\rm (b)}] $\frak pS\cap R=\frak p$ for all  $\frak p\in\Spec(R).$ 
\item [{\rm (c)}] If $\frak P\in\Spec(S)$ and $\frak p=\frak P\cap R,$ then $\frak P\in\Assh_S(S/\frak pS).$
\end{itemize}
\end{lemma}
	
\begin{proof} (a) Let $\dim (R)=d$ and $\frak p_0\subset  \frak p_1\subset \ldots \subset \frak p_d=\frak m$ be a saturated chain of prime ideals of $R$. By Remark \ref{R:3}(a), there exists $\frak P_0\in\Spec(S)$ such that $\frak P_0\cap R=\frak p_0.$ By Remark \ref{R:3}(b), there is a chain $\frak P_0\subset  \frak P_1\subset \ldots \subset \frak P_d=\frak n$  of prime ideals of $S$ such that $\frak P_i\cap R=\frak p_i$. Hence $d\leq \dim (S).$ Let $\dim(S)=d'$ and $\frak P_0\subset  \frak P_1\subset \ldots \subset \frak P_{d'}=\frak n$ be a saturated chain of prime ideals of $S$. Set $\frak p_i=\frak P_i\cap R.$ We get by Remark \ref{R:3}(c)  that $\frak p_i\neq \frak p_{i+1}$ for all $i$. Therefore, $d'\leq \dim (R).$  
				
(b) Let $\frak p\in\Spec(R).$ By  Remark \ref{R:3}(a),  there exists $\frak P\in\Spec(S)$ such that $\frak P\cap R=\frak p.$ Hence $\frak p\subseteq \frak pS\cap R\subseteq \frak P\cap R=\frak p.$ Therefore, $\frak pS\cap R=\frak p$.
		
(c) Since $\frak P\supseteq \frak pS$, there exists $\frak Q\in\Ass(S/\frak pS)$ such that $\frak P\supseteq \frak Q.$ Since $\frak p=\frak P\cap R$ and $S$ is finitely generated as an $R$-module, the canonical injection $R/\frak p\hookrightarrow S/\frak P$ is a module-finite extension. So we get by assertion (a) that 
$$\dim(R/\frak p)=\dim(S/\frak P)\leq \dim(S/\frak Q)\leq \dim(S/\frak pS).$$ 
Moreover, since $\frak pS\cap R=\frak p$  by assertion (b) and $S$  is finitely generated as an $R$-module, the canonical injection $R/\frak p\hookrightarrow S/\frak pS$ is a module-finite  extension. Hence $\dim(S/\frak pS)=\dim(R/\frak p)$ by assertion (a). Therefore, $$\dim(S/\frak P)= \dim(S/\frak Q)= \dim(S/\frak pS).$$ So, $\frak Q=\frak P\in\Assh_S(S/\frak pS).$
\end{proof}

For an  $S$-module $B$, we consider $B$ as an $R$-module defined by $\varphi$.   It is clear that if $\ell_R(B)<\infty$ then $\ell_S(B)<\infty.$ The following lemma shows that the converse statement is also true.
		
\begin{lemma}\label{L:4} Let $B$ be an $S$-module. The following statements are true.
\begin{itemize}
\item [{\rm (a)}] $\ell_R(S/\frak mS)<\infty.$ In particular, $\ell_R(S/\frak n)<\infty.$
\item [{\rm (b)}] $\ell_S(B)<\infty$   if and only if $\ell_R(B)<\infty.$ In this case, we have  $$\ell_R(B)=\ell_S(B)~\ell_R(S/\frak n).$$ 
\end{itemize}
\end{lemma}
\begin{proof} (a) Since $S$ is a finitely generated $R$-module, it follows that $S/\frak mS$ is a vector space of finite dimension over  $R/\frak m$. Hence $\ell_R(S/\frak mS)<\infty.$   Since  $\frak mS\subseteq \frak n$ by Remark \ref{R:3}(a), we have $\ell_R(S/\frak n)<\infty.$ 
		
(b) It is clear that if $\ell_R(B)<\infty$ then $\ell_S(B)<\infty$. Assume that $\ell_S(B)=t<\infty.$ If $t=0$ then there is nothing to do. Assume that    $t>0$. Let  $0=B_0\subset B_1\subset \ldots\subset B_t=B$ be a chain of $S$-submodules of $B$ such that $B_i/B_{i-1}\cong S/\frak n$ for all $i\geq 1.$ Then we get $$\ell_R(B)=\sum_{i=1}^{t}\ell_R(B_i/B_{i-1})=t\ell_R(S/\frak n)<\infty.$$	
\end{proof}

\section{Main results}	
	
\ \ \ \ \ The following theorem, which is the first main result of this paper, describes the transfer of the Artinian property as well as the behavior of attached primes under a module-finite extension. 

\begin{theorem}\label{T:1} The following statements are true. 
\begin{itemize}
\item [{\rm (a)}]  Let	 $B$  be  an $S$-module. Then $B$ is Artinian as an $S$-module if and only if $B$ is   Artinian as an $R$-module defined by $\varphi$. In this case we have $\dim_{S}(B)=\dim_R(B)$ and	$$\Att_S(B)\subseteq\bigcup_{\frak p\in\Att_R(B)}\Assh_S(S/\frak pS).$$	
								
\item [{\rm (b)}] If $A$ is an Artinian $R$-module then $A\otimes_RS$ is an Artinian $S$-module. In this case we have $\dim_S(A\otimes_RS)=\dim_R(A)$ and  					
$$\Att_S(A\otimes_RS)\subseteq\bigcup_{\frak p\in\Att_R(A)}\Assh_S(S/\frak pS).$$			
		
\item [{\rm (c) }] $\Att_S(B)=\bigcup_{\frak p\in\Att_R(B)}\Assh_S(S/\frak pS)$ for every Artinian $S$-module $B$ if and only if the induced map $\varphi^*:\Spec(S)\to\Spec(R)$ sending $\frak P$ to $\frak P\cap R$ is a bijection.				
\end{itemize}
\end{theorem}
		
\begin{proof} (a)  Assume that $B$ is Artinian as an $S$-module. We will use Melkersson's criterion (see Lemma \ref{L:1}) to show that $B$ is Artinian as an $R$-module. Let $b\in B.$ Since $B$ is an Artinian $S$-module, there exists $t\in\mathbb{N}$ such that $\frak n^tb=0.$  Since $\frak mS\subseteq \frak n$ by Remark \ref{R:3}(a), we have  $\frak m^tb=0$. Therefore, $B$ is $\frak m$-torsion as an $R$-module.  Now we prove that $(0:_B\frak m)$ is Artinian as an  $R$-module. Since $\frak mS\cap R=\frak m$ by Lemma \ref{L:3a}(b) and $S$ is finitely generated as an $R$-module, the  canonical injection map $R/\frak m\hookrightarrow S/\frak mS$ is a module-finite extension. Therefore, $\dim(S/\frak mS)=\dim(R/\frak m)=0$  by Lemma \ref{L:3a}(a). Hence $\dim _S(0:_B\frak mS)\leq\dim_S(S/\frak mS)=0.$ Since  $B$ is an Artinian $S$-module, $(0:_B\frak mS)$ is an Artinian  $S$-module. So, $\ell_S(0:_B\frak mS)<\infty$ by Lemma \ref{L:2}(b). Hence  $\ell_R(0:_B\frak m)<\infty$ by Lemma \ref{L:4}(b) and hence  $(0:_B\frak m)$ is an Artinian $R$-module. Therefore, $B$ is an Artinian $R$-module by Melkersson's criterion.	
		
Conversely, suppose that $B$ is Artinian as an $R$-module. Let $B_1\supseteq B_2\supseteq\ldots\supseteq B_t\ldots$ be a descending chain of $S$-submodules of $B.$ Then it is also a descending chain of $R$-submodules of $B.$ Since $B$ is Artinian as an $R$-module, this chain is stationary. Hence, $B$ is Artinian as an $S$-module.  
		
We prove that $\dim_S(B)=\dim_R(B)$ whenever $B$ is  Artinian as an $S$-module as well as an $R$-module. Set  $t=\dim_S(B)$ and $t'=\dim_R(B)$.	By Lemma \ref{L:2}(a), there exists  $\frak P\in\Att_{S}(B)$ such that $\dim(S/\frak P)=t.$ Set $\frak p =\frak P\cap R.$ Then $\frak p\in \Att_R(B)$ by Lemma \ref{L:2}(c).  Hence $\dim(R/\frak p)\leq t'$ by Lemma \ref{L:2}(a). Since $\frak p =\frak P\cap R$ and $S$ is finitely generated as an $R$-module, the canonical injection $R/\frak p\hookrightarrow S/\frak P$ is a module-finite extension. So, we get by Lemma \ref{L:3a}(a) that 
$$t=\dim(S/\frak P)=\dim(R/\frak p)\leq t'.$$
Conversely, it follows by Lemma \ref{L:2}(a) that there exists  an attached  prime $\frak p\in\Att_R(B)$ such that $\dim(R/\frak p)=t'.$ By Lemma \ref{L:2}(c), there exists  $\frak P\in\Att_S(B)$ such that $\frak p=\frak P\cap R.$ Hence the canonical injection $R/\frak p\hookrightarrow S/\frak P$ is a module-finite extension. Therefore,  we have by Lemma \ref{L:3a}(a) and Lemma \ref{L:2}(a) that  
$$t'=\dim(R/\frak p)=\dim(S/\frak P)\leq \dim_S(B)=t.$$
Therefore, $\dim_S(B)=\dim_R(B)$. 
			 
For the inclusion in statement (a), let $\frak P\in\Att_S(B).$ Set $\frak p=\frak P\cap R.$ Then $\frak p\in\Att_R(B)$ by Lemma \ref{L:2}(c) and  $\frak P\in\Assh_S(S/\frak pS)$  by Lemma \ref{L:3a}(c). Hence $$\Att_S(B)\subseteq\bigcup_{\frak p\in\Att_R(B)}\Assh_S(S/\frak pS).$$

(b) As $A$ is an Artinian $R$-module, we get by Lemma \ref{L:2a} that $A\otimes_RS$ is an Artinian $R$-module and $\Att_R(A\otimes_RS)=\Att_R(A)$. So, $A\otimes_RS$ is an Artinian $S$-module by assertion (a) and  $\dim_R(A\otimes_RS)=\dim_R(A)$  by Lemma \ref{L:2}(a).  Therefore, by applying assertion (a) for $S$-module $A\otimes_RS$, we have  $$\dim_S(A\otimes_RS)=\dim_R(A\otimes_RS)=\dim_R(A).$$ Moreover, by using the assertion (a)  again, we have
$$\Att_S(A\otimes_RS)\subseteq\bigcup_{\frak p\in\Att_R(A\otimes_RS)}\Assh_S(S/\frak pS)=\bigcup_{\frak p\in\Att_R(A)}\Assh_S(S/\frak pS).$$
	
(c) Suppose that $\Att_S(B)=\bigcup_{\frak p\in\Att_R(B)}\Assh_S(S/\frak pS)$ for every Artinian $S$-module $B$. By Remark \ref{R:3}(a), the induced map $\varphi^*:\Spec(S)\to\Spec(R)$ sending 
$\frak P$ to $\frak P\cap R$ is a surjection. Now we prove that $\varphi^*$ is injective. Let  
$\frak P, \frak Q\in\Spec(S)$ such that $\frak P\cap R=\frak Q\cap R=\frak p.$ Then $\frak P, \frak Q\in\Assh_S(S/\frak pS)$ by Lemma \ref{L:3a}(c). So we have $\dim(S/\frak P)=\dim(S/\frak Q)=\dim(S/\frak pS)$. Set $t=\dim(S/\frak pS).$ Set $B_1=H^t_{\frak n}(S/\frak P)$ and $B_2=H^t_{\frak n}(S/\frak Q).$ Then it follows by Lemma \ref{L:2}(d) that $B_1, B_2$ are Artinian $S$-modules and $\Att_S(B_1)=\{\frak P\}$,  $\Att_S(B_2)=\{\frak Q\}$. Note that $B_1, B_2$ are Artinian $R$-modules by assertion (a). So, we have by Lemma \ref{L:2}(c) that  
$\Att_R(B_1)=\{\frak P\cap R\}=\{\frak p\}$ and $\Att_R(B_2)=\{\frak Q\cap R\}=\{\frak p\}$. Therefore, we get by our assumption that 
$$\{\frak P\}=\Att_S(B_1)=\Assh_S(S/\frak pS)=\Att_S(B_2)=\{\frak Q\}.$$ 
Hence $\frak P=\frak Q$. It follows that $\varphi^*$ is a bijection.
	
Conversely, assume that $\varphi^*$ is a bijection. Let $B$ be an Artinian $S$-module. From the assertion (a),  it is enough to prove that $\Att_S(B)\supseteq \bigcup_{\frak p\in\Att_R(B)}\Assh_S(S/\frak pS)$. Let $\frak p\in \Att_R(B)$ and $\frak P\in\Assh_S(S/\frak pS)$. Then $\dim(S/\frak P)=\dim(S/\frak pS)$. Set $\frak q=\frak P\cap R.$ Then we have by Lemma \ref{L:3a}(c) that $\frak P\in\Assh(S/\frak qS).$ Hence $\dim(S/\frak P)=\dim(S/\frak qS)$. It follows that $\dim(S/\frak qS)=\dim(S/\frak pS).$ Note that $\frak pS\cap R=\frak p$ and $\frak qS\cap R=\frak q$ by Lemma \ref{L:3a}(b). Therefore, the ring injections $R/\frak p\hookrightarrow S/\frak pS$ and $R/\frak q\hookrightarrow S/\frak qS$ are module-finite extensions. Hence 
$$\dim (R/\frak q)=\dim(S/\frak qS)=\dim(S/\frak pS)=\dim (R/\frak p)$$ 
by Lemma \ref{L:3a}(a).	Note that $\frak pS\subseteq \frak P$. Hence  $\frak p=\frak pS\cap R\subseteq \frak P\cap R=\frak q.$   It follows that $\frak p=\frak q.$ Because $\frak p\in\Att_R(B)$, we get by Lemma \ref{L:2}(c) that there exists a prime ideal $\frak P'\in \Att_S(B)$ such that $\frak P'\cap R=\frak p.$ It follows that  $\frak P'\cap R=\frak p=\frak q=\frak P\cap R.$ Since $\varphi^*: \Spec(S)\to \Spec(R)$ is an injection by our assumption, we have $\frak P=\frak P'\in\Att_S(B).$ 
\end{proof}
	 
From the statement of Theorem \ref{T:1}(b), it is natural to ask the following question:  {\it Let $A$ be an $R$-module such that $A\otimes_RS$ is an Artinian $R$-module. Is $A$ an Artinian $R$-module?}
 
It should be mentioned that,  for a submodule $A'$ of $A$,  the natural homomorphisms $A\to A\otimes_RS$ and $A'\otimes_RS\to A\otimes_RS$ are not necessarily injective, see  Example \ref{E:1}. Although $A\otimes_RS$ is Artinian and it is a quotient of $A^t$ (where $t=\ell_R(S/\frak  mS))$, we do not know whether $A^t$ is Artinian. So, in the general case,  it seems difficult to determine if $A$ is Artinian or not.
 
 Here are some cases where the answer is positive. 
 
 \begin{lemma} \label{L:5} Let $A$ be an $R$-module such that $A\otimes_RS$ is an Artinian $R$-module. If one of the following conditions is satisfied then $A$ is Artinian.
 \begin{itemize}
 \item [{\rm (a)}]	$\varphi: R\hookrightarrow S$ is flat. 
 	
 \item [{\rm (b)}]   $R$ is a direct summand of $S$ as an $R$-module.
 	
 \item [{\rm (c)}]  $\frak mA=0$, i.e. $A$ is a vector space over $R/\frak m$. 
 \end{itemize}
 \end{lemma}

\begin{proof} (a) If $A$ is not Artinian then there exists a descending chain $A_1\supset A_2\supset \ldots $ of submodules of $A$ which is not stationary. As $S$ is flat over $R$, $A_1\otimes_RS\supset A_2\otimes_RS\supset \ldots $ is a descending chain of submodules of $A\otimes_RS$ which is not stationary. This is impossible.
 		
(b)  Since $R$ is a direct summand of $S$, we can consider $A$ as a direct summand of $A\otimes_RS$, so $A$ is Artinian.
 		
(c) Set $k=R/\frak m$. Since $\frak mA=0,$ we have $A\otimes_RS\cong A\otimes_kS/\frak mS$.  Hence $A$ is Artinian.
 \end{proof}

 The notion of Cohen-Macaulay module is a central object of research in commutative algebra.  There are at least two different ways to generalize this notion; one is the notion of generalized Cohen-Macaulay module introduced by Cuong-Schenzel-Trung \cite{CST} and the other one is the notion of sequentially Cohen-Macaulay module  introduced by  R. P. Stanley \cite{St} in the graded setting  and by P. Schenzel \cite{Sch} in the local setting. In a natural way, the notions of sequentially generalized Cohen-Macaulay module and  sequentially $1$-Cohen-Macaulay module were defined respectively in \cite{CN} and \cite{LLN}. 
 	
Let $M$ be a finitely generated $R$-module of dimension $d$. Set $\frak a_i(M):=\Ann_RH^i_{\frak m}(M)$ for $i\leq d$ and $\frak a(M)=\frak a_0(M)\ldots \frak a_{d-1}(M).$ If we stipulate the dimension of zero module to be $-1$, then $M$ is Cohen-Macaulay (resp. generalized Cohen-Macaulay, $1$-Cohen-Macaulay) if and only if $\dim (R/\frak a(M))=-1$ (resp. $\dim (R/\frak a(M))\leq 0$, $\dim (R/\frak a(M))\leq 1$). 	
Let $H^0_{\frak m}(M)=D_t\subset \ldots \subset D_1\subset D_0=M$ be the dimension filtration of $M$, i.e. $D_{i+1}$ is the largest submodule of $M$ of dimension less than $\dim_R(D_i)$ for all $i<t$, see \cite{Sch}. Note that the dimension filtration of $M$ always exists uniquely.  We say that $M$ is {\it sequentially Cohen-Macaulay} (resp. {\it sequentially generalized Cohen-Macaulay}, {\it sequentially $1$-Cohen-Macaulay}) if each quotient $D_i/D_{i+1}$ is Cohen-Macaulay (resp. generalized Cohen-Macaulay, $1$-Cohen-Macaulay).  The structure of sequentially  Cohen-Macaulay modules,  sequentially generalized Cohen-Macaulay modules and sequentially $1$-Cohen-Macaulay modules  are  extensively studied, see  \cite{CSSS, CN, CST, GN, LLN, LN, NDC, Sch, St, TPDA}. 
 	
It is useful to know whether a local ring is a quotient of a Cohen-Macaulay local ring because this relates closely to various subjects such as Macaulayfication, Faltings' Annihilator Theorem, Shifted Principles for local cohomology modules, see \cite{CC, Fa, NQ}. 
 	
Next, we apply Theorem \ref{T:1} to describe the behavior of local cohomology modules as well as the structure of certain finitely generated modules under module-finite extensions. We also show that the property of being a quotient of a Cohen-Macaulay local ring can  transfer between $R$ and $S.$ 
 		 	
\begin{theorem}\label{T:2} Let $N$ be a finitely generated $S$-module. The following statements hold true. 
\begin{itemize}			
\item [{\rm (a)}] For each integer $i$, there is an isomorphism $H^i_{\frak n}(N)\cong  H^i_{\frak m}(N)$ of Artinian $R$-modules. Moreover, if $H^i_{\frak n}(N)\neq 0$ then  
$\dim_S(H^i_{\frak n}(N))=\dim_R(H^i_{\frak m}(N)).$
\item [{\rm (b)}] $N$ is Cohen-Macaulay (resp. generalized Cohen-Macaulay, $1$-Cohen-Macaulay,  sequentially  Cohen-Macaulay, sequentially  generalized Cohen-Macaulay, sequentially  $1$-Cohen-Macaulay)  as an $S$-module if and only if so is $N$ as an $R$-module.
\item [{\rm (c)}]  $R$ is a quotient of a Cohen-Macaulay local ring if and only if so is  $S$.
\end{itemize} 		
\end{theorem}
 	
\begin{proof} (a) Since $\varphi: R\hookrightarrow S$ is a module-finite extension, $N$ is finitely generated as an $R$-module defined by $\varphi$. Let $i$ be an integer. Because  $H^i_{\frak n}(N)$ is an Artinian $S$-module, we get by Theorem \ref{T:1}(a) that $H^i_{\frak n}(N)$ is Artinian as an $R$-module and  $\dim_S(H^i_{\frak n}(N))=\dim_R(H^i_{\frak n}(N))$. As $\ell_R(S/\frak mS)<\infty$ by Lemma \ref{L:4}(a), we have $\ell_S(S/\frak mS)<\infty$, i.e. $\dim(S/\frak mS)=0.$ So,  we get by Independence Theorem \cite[Theorem 4.2.1]{BS} that there  is an isomorphism  $$H^i_{\frak n}(N)=H^i_{\frak mS}(N)\cong  H^i_{\frak m}(N)$$  of $R$-modules.  If $H^i_{\frak n}(N)\neq 0$ then this isomorphism gives
$$\dim_S(H^i_{\frak n}(N))=\dim_R(H^i_{\frak n}(N))=\dim_R(H^i_{\frak m}(N)).$$
 		 
(b) We get by assertion (a) that $H^i_{\frak n}(N)\neq 0$ if and only if $H^i_{\frak m}(N)\neq 0$ for every integer $i$. Hence  $\dim_S(N)=\dim_R(N)$ and  $\depth_S(N)=\depth_R(N)$. So, $N$ is Cohen-Macaulay as an $S$-module if and only if $N$ is Cohen-Macaulay as an $R$-module.	
 		 
It follows by assertion (a) that  $\dim_S(H^i_{\frak n}(N))\leq 0$ for all $i<\dim_S(N)$ if and only if $\dim_R(H^i_{\frak m}(N))\leq 0$ for all $i<\dim_R(N).$ It means that $N$ is generalized Cohen-Macaulay as an $S$-module if and only if $N$ is generalized Cohen-Macaulay as an $R$-module. By the same arguments, it follows that  $N$ is $1$-Cohen-Macaulay as an $S$-module if and only if $N$ is $1$-Cohen-Macaulay as an $R$-module.
 		 		 
Next, let $H^0_{\frak n}(N)=D_t\subset \ldots \subset D_1\subset D_0=N$ be the dimension filtration of $N$ as an $S$-module.  Since $\varphi: R\hookrightarrow S$ is  a module-finite extension, we get by \cite[Theorem 3.3]{TPDA} that this filtration is also the dimension filtration of $N$ as an $R$-module defined by $\varphi$. Therefore, $N$ is  sequential Cohen-Macaulay as an $S$-module if and only if each $D_i/D_{i+1}$ is Cohen-Macaulay as an $S$-module, if and only if  each $D_i/D_{i+1}$ is Cohen-Macaulay as an $R$-module, if and only if $N$ is sequentially Cohen-Macaulay as an $R$-module. The rest of the statement (a) follows by the similar arguments.
 		
(c) We first recall a criterion of quotients of Cohen-Macaulay local rings in \cite{CC}. Let $M$ be a finitely generated $R$-module, $\dim_R(M)=d$.  Following Cuong \cite{C}, a system of parameters $x_1, \ldots , x_d$ of $M$ is  a {\it $p$-standard system of parameters}  of $M$ if $x_d\in \frak a(M)$ and $x_i\in\frak a(M/(x_{i+1}, \ldots , x_d)M)$ for all $i$. Here, $\frak a(M):=\Ann_R(H^0_{\frak m}(M))\ldots \Ann_R(H^{d-1}_{\frak m}(M)).$  By \cite[Theorems 1.2, 1.3]{CC}, the ring $R/\Ann_R(M)$ is a quotient of a Cohen-Macaulay local ring if and only if $M$ admits  a $p$-standard system of parameters.
 		
 We prove the statement (c). Let $\dim (R)=n$. Then $\dim (S)=n$.  To avoid confusion, set 
\begin{align}\frak a_S(S):&=\Ann_S(H^0_{\frak n}(S))\ldots \frak \Ann_S(H^{n-1}_{\frak n}(S))\notag\\
\frak a_R(S):&=\Ann_R(H^0_{\frak m}(S))\ldots \frak \Ann_R(H^{n-1}_{\frak m}(S)).\notag\end{align} 
Then we have by assertion (a) that $\dim_S(H^i_{\frak n}(S))=\dim_R(H^i_{\frak m}(S))$ for all integers $i<n.$ Therefore, $\dim (S/\frak a_S(S))=\dim (R/\frak a_R(S)).$
 	
Suppose that $R$ is a quotient of a Cohen-Macaulay local ring. Then $\dim_R(H^i_{\frak m}(S))\leq i$ for all $i$ by Lemma \ref{L:2}(a),(e). Hence, $\dim (S/\frak a_S(S))<n.$ By Prime Avoidance, there exists a parameter $s_n$ of $S$ such that $s_n\in \frak a_S(S)$. By the same reasons,  $\dim (S/\frak a_S(S/s_nS))<n-1.$ So, there exists a parameter $s_{n-1}$ of $S/s_nS$ such that $s_{n-1}\in \frak a_S(S/s_nS)$. Continue this process, we get a $p$-standard system of parameters $s_1, \ldots , s_n$ of $S$. Hence, the ring $S$ is a quotient of a Cohen-Macaulay local ring.
 	
Conversely, suppose that $S$ is a quotient of a Cohen-Macaulay local ring. By Lemma \ref{L:2}(a),(e), we have $\dim_S(H^i_{\frak n}(S))\leq i$ for all $i$. Hence $\dim (R/\frak a_R(S))<n.$ So, by the same arguments as in the above, $R$-module $S$ admits a $p$-standard system of parameters. Hence $R/\Ann_R(S)$ is a quotient of a Cohen-Macaulay local ring. Note that $\Ann_R(S)=0.$ So, $R$ is a quotient of a Cohen-Macaulay local ring.  		
\end{proof}
 	
We continue with an application of Theorem \ref{T:2} in  studying the structure of the idealization introduced by Nagata \cite{Na}.  Let $M$ be a finitely generated $R$-module. We provide a multiplication on the additive group $R\oplus M$    $$(a, x).(b, y) = (ab, ay+bx)$$ for all $(a, x), (b, y)\in R\oplus M$, then $R\oplus M$ forms a Noetherian local ring with the unique maximal ideal $\mathfrak{m} \oplus M$ and $\dim(R\ltimes M)=\dim(R)$. This local ring is called the {\it idealization} of $M$ over $R$ and denoted by $R\ltimes M$. There is a canonical   projection $R\ltimes M \to R$ sending $(a,x)$ to $a$ and a canonical  injection $\sigma: R\hookrightarrow R\ltimes M$ sending $a$ to $(a,0).$ These maps are local homomorphisms.  We note that the canonical  injection $\sigma$  is a module-finite extension. The idealization has  important applications in commutative algebra, so its structure has attracted the interest of mathematicians, for example see \cite{AW, CNN, GoKu, Na}.
	
\begin{corollary} \label{C:1}  Let $M\neq 0$ be a finitely generated $R$-module. Let $S=R\ltimes M$ be the idealization of $M$ over $R$. Let $\dim (R)=n$. Then   
\begin{itemize}
\item[\rm{(a)}] $S$ is a Cohen-Macaulay ring if and only if $S$ is a Cohen-Macaulay $R$-module, if and only if   both $R, M$ are Cohen-Macaulay of dimension $n$.
				
\item[\rm{(b)}] $S$ is a generalized Cohen-Macaulay ring if and only if $S$ is a generalized Cohen-Macaulay $R$-module, if and only if  $R$ is generalized Cohen-Macaulay and either $\dim_R(M)\leq 0$ or  $M$ is generalized Cohen-Macaulay of dimension $n$.
				
\item[\rm{(c)}] $S$ is an $1$-Cohen-Macaulay ring if and only if $S$ is an $1$-Cohen-Macaulay $R$-module, if and only if  $R$ is $1$-Cohen-Macaulay and either $\dim_R(M)\leq 1$ or  $M$ is $1$-Cohen-Macaulay of dimension $n$. 							
\end{itemize}
\end{corollary}
		
\begin{proof} We have $\dim(S)=\dim(R)=n$. Set $\dim_R(M)=d.$ Note that $0\leq d\leq n$ and  $S=R\oplus M$ as $R$-modules.  So, we have  an isomorphism of $R$-modules for each integer $i\leq n$ $$H_{\frak m}^i(S)\cong H^i_{\frak m}(R)\oplus H^i_{\frak m}(M).$$ 	
 Note that $H^d_{\frak m}(M)\neq 0$ and $\dim_RH^d_{\frak m}(M)=d$ by Lemma \ref{L:2}. So,  by the above isomorphism, $S$ is a Cohen-Macaulay $R$-module if and only if $H^i_{\frak m}(R)=0$ and $H^i_{\frak m}(M)=0$ for all $i<n$, if and only if  $R$ and $ M$ are Cohen-Macaulay $R$-modules of dimension $n$. Next, $S$ is a generalized Cohen-Macaulay  $R$-module if and only if $\dim_R(H^i_{\frak m}(R))\leq 0$ and $\dim_R(H^i_{\frak m}(M))\leq 0$  for all $i<n$, if and only if $R$ is generalized Cohen-Macaulay  and either $\dim_R(M)\leq 0$ or  $M$ is generalized Cohen-Macaulay of dimension $n$. Similarly,  
  $S$ is an $1$-Cohen-Macaulay $R$-module if and only if  $\dim_R(H^i_{\frak m}(R))\leq 1$ and $\dim_R(H^i_{\frak m}(M))\leq 1$ for all $i<n$, if and only if $R$ is $1$-Cohen-Macaulay and either  $\dim_R(M)\leq 1$ or  $M$ is $1$-Cohen-Macaulay  of dimension $n$. The rest  statement of (a), (b), (c) follows by  Theorem \ref{T:2}(b). 
\end{proof}

\begin{corollary} \label{C:2} Let $M$ be a finitely generated $R$-module. Then $R\ltimes M$ is sequentially  Cohen-Macaulay (resp. sequentially  generalized Cohen-Macaulay, sequentially  $1$-Cohen-Macaulay) as a ring  if and only if  so is $R\ltimes M$ as an $R$-module, if and only if so are $R$ and $M$.  
\end{corollary}

\begin{proof}  	Set $n=\dim (R)$, $d=\dim_R(M)$ and $S=R\ltimes M$. Then $d\leq n.$ Note that $S=R\oplus M$ as $R$-modules and $\dim_R(R\oplus M)=n$. Therefore, by Theorem \ref{T:2}(b), it is enough to prove that $R\oplus M$ is a sequentially  Cohen-Macaulay (resp. sequentially  generalized Cohen-Macaulay, sequentially  $1$-Cohen-Macaulay)  $R$-module if and only if so are $R$ and $M$. We  proceed by induction on $n$. If  $n=0$ then $R, M, R\oplus M$ are Cohen-Macaulay, so the result  is clear. 
		
Let $n\geq 1.$ Let $S_1$ be the largest $R$-submodule of $R\oplus M$ of dimension less than $n$.  Let $M_1$ be the largest submodule of $M$  of dimension less than $d$ and  $R_1$ the largest $R$-submodule of $R$  of dimension less than $n$.  If $d=n$, we can check that $S_1=R_1\oplus M_1$, therefore $S/S_1\cong R/R_1\oplus M/M_1$. If $d<n$ then $S_1=R_1\oplus M$, so we get   $S/S_1\cong R/R_1$. As $\dim_R(S_1)<n,$ we get by induction that $S_1$ is a sequentially  Cohen-Macaulay (resp. sequentially  generalized Cohen-Macaulay, sequentially  $1$-Cohen-Macaulay) $R$-module if and only if so are $R_1$ and $M_1$. 
		
If $d<n$ then $S/S_1$ is a Cohen-Macaulay (resp. generalized Cohen-Macaulay, $1$-Cohen-Macaulay) $R$-module if and only if so is  $R/R_1$.  So, the result follows.	
		
Assume that $d=n$. Note that $\dim_R(R/R_1)=\dim_R(M/M_1)=n$. Moreover, $$H_{\frak m}^i(S/S_1)\cong H^i_{\frak m}(R/R_1)\oplus H^i_{\frak m}(M/M_1)$$ for all $i<n$. It follows that $S/S_1$ is a Cohen-Macaulay (resp. generalized Cohen-Macaulay, $1$-Cohen-Macaulay) $R$-module if and only if so are $R/R_1$ and $M/M_1$. Thus, we also get the result in this case. 		
\end{proof}

We mention that  the characterization for the sequentially Cohen-Macaulayness in Theorem \ref{T:2}(b) and Corollary \ref{C:2} is already provided in \cite{TPDA}. 
		
Let P be one of the properties: Cohen-Macaulayness, generalized Cohen-Macaulayness, $1$-Cohen-Macaulayness,  sequentially  Cohen-Macaulayness, sequentially  generalized Cohen-Macaulayness, sequentially  $1$-Cohen-Macaulayness. For a finitely generated $R$-module $M$, it is natural to ask whether the property P on $S$-module $M\otimes_RS$ is inherited by $M$, and conversely whether the property P lifts from $M$ to $M\otimes_RS$. The answer is negative even in the case where $M=R.$ Below we use  Corollaries \ref{C:1}, \ref{C:2} to clarify this situation.
		
\begin{example} \label{E:2} {\rm Let $P$ be a property as in the above. Then}
\begin{itemize}
\item[\rm{(a)}] {\rm There exists a module-finite extension $R\hookrightarrow S$ of local rings such that  $R$ satisfies the property P, but $S$ does not satisfy the property P.}
\item[\rm{(b)}] {\rm There exists a module-finite extension $R\hookrightarrow S$ of local domains such that $S$ satisfies the property P, but $R$ does not satisfy the property P.}
\end{itemize}

\begin{proof} (a) Let $d\geq 1$ be an integer. Let $R=K[[x_1,\ldots,x_d]]$ be the formal power series ring of $d$ variables over a field $K$. Then $R$ is Cohen-Macaulay of dimension $d$.	We set $M_i=R/(x_1, \ldots , x_i)R$ and $M'_i=(x_1, \ldots , x_i)R$ for $1\leq i\leq d$.
			 
We first choose $S=R\ltimes M_i$. We note by Corollary \ref{C:2} that  $S$ is always sequentially Cohen-Macaulay for any $i$. However, it follows by Corollary \ref{C:1} that if $i=d$ then $S$ is not Cohen-Macaulay for any integer $d\geq 1$; if $i=d-1$ then $S$ is not generalized Cohen-Macaulay for any integer $d\geq 2$; if $i=d-2$ then $S$ is not $1$-Cohen-Macaulay for any integer $d\geq 3.$  
			
Next, we choose $S=R\ltimes M'_i$. Then $0\subset S$ is the dimension filtration of $S$ and $0\subset M'_i$ is the dimension filtration of $M'_i$. Note that if $i=d$ then $M'_i$ is not Cohen-Macaulay for any integer $d\geq 2$, hence $M'_i$ is not sequentially Cohen-Macaulay, and hence $S$ is not sequentially Cohen-Macaulay by Corollary \ref{C:2}. Similarly, by Corollary \ref{C:2}, if $i=d-1$ then $M'_i$ is not generalized Cohen-Macaulay for any integer $d\geq 3$, hence $S$ is not sequentially generalized Cohen-Macaulay; if $i=d-2$ then $M'_i$ is not $1$-Cohen-Macaulay for any integer $d\geq 4,$ hence $S$ is not sequentially $1$-Cohen-Macaulay.
			
(b) Let $K$ be an infinite field. Let $V$ be the algebraic variety in $\mathbb{A}^n$ defined by $$x_1=s^4;~x_2=s^3t;~x_3=st^3;~x_4=t^4;~x_5=s^2t^2v;~x_6=v.$$ Then using the proof of \cite[Example 5.4(i)]{NM}, it follows that the Macaulayfication of $V$ is the algebraic variety $W$ defined by $x_1=s^4;~x_2=s^3t;~x_3=st^3;~x_4=t^4;~x_5=s^2t^2;~x_6=v,$  and we have the following exact sequence $$0\to K[V]\to K[W]\to L\to 0,$$ where $K[V],~ K[W]$ are respectively the coordinate rings of $V,~ W$ and $L$ is a non-zero module of dimension $0.$   Let $R_1=K[[s^4, s^3t, st^3, t^4, s^2t^2v, v]]$ and $S_1=K[[s^4, s^3t, st^3, t^4, s^2t^2, v]]$. Then $R_1$ is a local domain with the unique maximal ideal  $\frak m_1=(s^4, s^3t, st^3, t^4, s^2t^2v, v)$ and $\dim(R_1)=3$. Also,  $S_1$ is a local domain, $\frak n_1=(s^4, s^3t, st^3, t^4, s^2t^2, v)$  is the unique maximal ideal of $S_1$ and $\dim(S_1)=3$. Note that  $R_1\hookrightarrow S_1$  is a module-finite extension. It follows by  the above exact sequence that  $S_1$ is Cohen-Macaulay, $H^1_{\frak m_1}(R_1)\neq 0$, $\dim_{R_1}(H^1_{\frak m_1}(R_1))=0$  and $H^i_{\frak m_1}(R_1)=0$ for $i=0$ and $i=2.$ Therefore, $R_1$  is generalized Cohen-Macaulay and $R_1$  is not Cohen-Macaulay.
			
Let $d\geq 3$ be an integer. Let $R=R_1[[w_1, \ldots , w_{d-3}]]$ and $S=S_1[[w_1, \ldots , w_{d-3}]]$ be the formal power series rings of $d-3$ variables over $R_1$ and $S_1$, respectively.  Then $R$ and $S$ are local domains, $\dim(R)=\dim (S)=d$ and  $\frak m=(\frak m_1,w_1, \ldots , w_{d-3})$ is the maximal ideal of $R$, and  $\frak n=(\frak n_1,w_1, \ldots , w_{d-3})$ is the unique maximal ideal of $S$. Note that  $R\hookrightarrow S$ is a module-finite extension.    Since $S_1$ is Cohen-Macaulay, it follows that $S$ is Cohen-Macaulay for any $d\geq 3.$ We note that $$\dim_R(H^{d-2}_{\frak m}(R))=d-3+\dim_{R_1}(H^1_{\frak m_1}(R_1))=d-3.$$ Note that $0\subset R$ is the dimension filtration of $R$. Therefore, $R$ is neither Cohen-Macaulay nor sequentially Cohen-Macaulay for all $d\geq 3$; $R$ is neither generalized Cohen-Macaulay nor sequentially generalized Cohen-Macaulay for all $d\geq 4$; and $R$ is neither $1$-Cohen-Macaulay nor sequentially $1$-Cohen-Macaulay for all $d\geq 5$.    	
\end{proof}		
\end{example}

\end{document}